\documentclass{article}
\usepackage[utf8]{inputenc}
\usepackage[a4paper]{geometry}
\usepackage{amsmath}
\usepackage{amsthm}
\usepackage{color}
\newtheorem{theorem}{Theorem}[section]
\newtheorem{remark}{Remark}[section]
\newtheorem{corollary}{Corollary}[theorem]

\theoremstyle{definition}
\newtheorem{definition}{Definition}[section]
\newtheorem{example}{Example}[section]
\usepackage{amssymb}
\usepackage{graphicx}
\title{Primal Topological Spaces}
\author{Santanu Acharjee$^1$, Murad \"{O}zko\c{c}$^2$,  Faical Yacine Issaka$^3$\\
$^1$Department of Mathematics\\
Gauhati University\\
Guwahati-781014, Assam, India\\
$^2$Mu\u{g}la S{\i}tk{\i} Ko\c{c}man University\\
Faculty of Science\\ Department of Mathematics \\ 48000,
Mente\c{s}e-Mu\u{g}la, Turkey\\
$^3$Mu\u{g}la S{\i}tk{\i} Ko\c{c}man University\\
Graduate School of Applied and Natural Sciences\\ Mathematics, 48000, Mente\c{s}e-Mu\u{g}la, Turkey\\
e-mails: $^1$sacharjee326@gmail.com, $^2$murad.ozkoc@mu.edu.tr\\
$^3$faicalyacine@gmail.com\\
}
\date{}
\begin{document}
\maketitle

\begin{abstract}
The purpose of this paper is to introduce a new structure `primal'. Primal is dual to grill.  Like ideal,  dual of filter, this new structure also generates a  new topology named `primal topology'.  We introduce a new operator using primal, which satisfies Kuratowski's closure axioms. Mainly, we prove that primal topology is finer than the topology of a primal topological space. We provide structure of base of primal topology and prove other fundamental results related to this new structure. \\  

{\bf 2020 AMS Classifications:} 54A05;  54B99; 94A60.\\


\end{abstract}

\section{Introduction}
Topology is one of the branches of mathematics which is highly applicable \cite{Wl}. Due to its applicability in both science and social science, several new ideas have been developed in topology  with classical structures. Kuratowski introduced the idea of ideal from filter \cite{Ku}. One may consider ideal as the dual of filter. Similarly, one of the classical structures of topology is grill. The definition of grill was introduced by Choquet \cite{ch} in the year 1947.\\

Later, Thron \cite{wj} introduced proximity structures in grills. In 1977, Chattopadhyay and Thron \cite{kc} extended ideas of closure space with grills. Moreover, Chattopadhyay et al. \cite{ch1} extended ideas of grills to study merotopic spaces. Since then, the structure of grill has been highly used in topology. Roy and Mukherjee \cite{b1,b2,b3} studied various topological properties with grills. Operators based on grills were introduced by Roy et al. \cite{b4}, Nasef and  Azzam \cite{Na}, and many others.   Modak \cite{m1, m2} studied grill-filter space and related properties. Hosny \cite{ro} studied grill structures in $\delta$-set. Cluster systems via grills were studied in \cite{rd}. Moreover, various advanced results with grills were studied in \cite{ az, ta1, ta2, mn, lo} and many others. But, it is important to note that the literature on grill structures are less comparison to filter, ideal, etc. Moreover, interdisciplinary applications of grills are rare to be found.  Janković and  Hamlett \cite{Ja} introduced a new topological space using ideal from a given topology of a topological space. Since, primal is also dual structure of grill, thus we are motivated by Janković and  Hamlett \cite{Ja} to introduce a new topology using primal.\\ 

In this paper, we  introduce  the dual structure of grill named `primal'. Moreover, we introduce a new topology named ``primal topology" and study several fundamental properties. In 1990, Isham \cite{Is} connected general topology with quantum topology. In this paper, we also  observe scopes of some quantum behavior of a new structure   $(. )^\diamond$ introduced in Definition 3.3. Similarly, uncertain behaviors can be observed in Theorem 3.10 and results next to it. As of our knowledge is concerned,  this paper is the first paper in the literature of topology to introduce dual structure of grill, thus we are  confined to study fundamental results only related to primal.\\

\section{Preliminaries}

Throughout this present paper, $(X,\tau)$ and $(Y,\sigma)$ (briefly, $X$ and $Y$) represent topological spaces unless otherwise stated. For any subset $A$ of a space $X$, $cl(A)$ and $int(A)$ denote  closure and t interior of $A$, respectively. The powerset of a set $X$ will be denoted by $2^X.$ The family of all open neighborhoods of a point $x$ of $X$ is denoted by $\tau(x).$ Now, we procure the following definition of grill. \\

\begin{definition} \cite{ch}
A family $\mathcal{G}$ of $2^X$ is called a grill  on $X$ if $\mathcal{G}$ satisfies the following conditions:\\

$(1)$ $\emptyset\notin \mathcal{G},$\\

$(2)$ if $A\in\mathcal{G}$ and $A\subseteq B,$ then $B\in \mathcal{G}$,\\

$(3)$ if $A\cup B\in \mathcal{G},$ then $A\in\mathcal{G}$ or $B\in\mathcal{G}.$\\
\end{definition}

Since, primal is dual to grill, thus we skip to add other results on grill, otherwise it will unnecessarily increase the volume of this paper. 

\section{Primal and a new topological structure}

In this section, we introduce a new structure in topology. This new structure is called primal. Primal is a dual structure of grill.  Now, we have the following definition of primal.\\

\begin{definition}
Let $X$ be a nonempty set. A collection $\mathcal{P}\subseteq   2^X$  is called a primal on $X$ if it satisfies the following conditions:\\

(i) $X \notin \mathcal{P}$,\\

(ii) if $A\in \mathcal{P}$ and $B\subseteq A$, then $B\in \mathcal{P}$,\\

(iii) if $A\cap B\in \mathcal{P}$, then $A\in \mathcal{P}$ or $B\in \mathcal{P}$.\\
\end{definition}

\begin{corollary}
Let $X$ be a nonempty set. A collection $\mathcal{P}\subseteq   2^X$  is a primal on $X$ if and only if it satisfies the following conditions:\\

(i) $X \notin \mathcal{P}$,\\

(ii) if $B\notin \mathcal{P}$ and $B\subseteq A$, then $A\notin \mathcal{P}$,\\

(iii) if $A\notin \mathcal{P}$ and $B\notin \mathcal{P},$ then $A\cap B\notin \mathcal{P}.$
\end{corollary}

Now, let us consider the following examples.

\begin{example}
Let $X=\{a,b\}.$ Then, all primals defined on $X$ are $\mathcal{P}_1=\emptyset,$ $\mathcal{P}_2=\{\emptyset,\{a\}\},$ $\mathcal{P}_3=\{\emptyset,\{b\}\}$ and $\mathcal{P}_4=\{\emptyset,\{a\},\{b\}\}.$
\end{example}

\begin{example}
Let $X$ be a nonempty set. It is obvious that the family $\mathcal{P}=2^X\setminus \{X\}$ is a primal on $X,$ where $2^X$ denotes the powerset of $X.$
\end{example}

\begin{corollary}
It is not difficult to see that $2^n$ primals can be written on a set with $n$ elements.
\end{corollary}

\begin{theorem}
Let $\mathcal{G}$ be a grill on $X$. Then, $\{A| A^{c}\in \mathcal{G}\}$ is a primal on $X$.
\end{theorem}

\begin{proof}
Let $\mathcal{G}$ be a grill of $X$ and $\mathcal{P}=\{A| A^{c}\in \mathcal{G}\}$. Then, we are to show that $\mathcal{P}$ is a primal. \\

$(i)$ Since $\emptyset \notin \mathcal{G}$, thus $X \notin \mathcal{P}$.\\

$(ii)$ Let $A\in \mathcal{P}$ and $B\subseteq A$. Then, $A^{c}\subseteq B^{c}$. Since $A^{c}\in \mathcal{G}$, thus $B^{c}\in \mathcal{G}$. Hence, $B\in \mathcal{P}$.\\

$(iii)$ Let $A\cap B\in \mathcal{P}$. Then, $A^{c}\cup B^{c}=(A\cap B)^{c}\in \mathcal{G}$. Therefore, we get $A^{c} \in \mathcal{G}$ or $B^{c} \in \mathcal{G}.$ Hence, $A \in \mathcal{P}$ or $B \in \mathcal{P}$. Thus, $\mathcal{P}$ is a primal on $X.$
\end{proof}

\begin{theorem}
Let $\mathcal{P}$ and $\mathcal{Q}$ be two primals on $X$. Then, $\mathcal{P\cup Q}$ is a primal on $X$.
\end{theorem}

\begin{proof}
$(i)$ Since $\mathcal{P}$ and $\mathcal{Q}$ be two primals on $X$, thus we get $X\notin \mathcal{P}$ and  $X\notin \mathcal{Q}$. Hence, $ X\notin \mathcal{P\cup Q}$.\\

$(ii)$ Let $A\in \mathcal{P\cup Q}$ and $B\subseteq A$. Then, $A\in \mathcal{P}$ or $A\in \mathcal{Q}$.  Then, $B\in \mathcal{P}$ or $B\in \mathcal{Q}$. It yields  $B\in \mathcal{P\cup Q}$.\\

$(iii)$ Let $A\cap B\in \mathcal{P\cup Q}$. Then, $A\cap B\in \mathcal{P}$ or $A\cap B\in \mathcal{Q}$. If $A\cap B\in \mathcal{P}$, then either $A \in \mathcal{P}$ or $ B\in \mathcal{P}$. Again,  if $A\cap B\in \mathcal{Q}$, then  either $B\in \mathcal{Q}$ or $ B\in \mathcal{Q}$. Then, obviously $A \in \mathcal{P\cup Q}$ or $ B\in \mathcal{P\cup Q}$.\\ 

Thus, $\mathcal{P\cup Q}$ is a primal on $X$.
\end{proof}

\begin{remark}
The intersection of two primals defined on $X$ need not be a primal on $X$ as shown by the following example.
\end{remark}

\begin{example}
Let $X=\{a,b\}$ with primals  $\mathcal{P} = \{\emptyset, \{a\}\}$ and $\mathcal{Q} = \{\emptyset, \{b\}\}$ on $X$. Then, $\mathcal{P\cap Q}=\{\emptyset\}$ is not a primal on $X$ since $\{a\}\cap\{b\}=\emptyset\in \mathcal{P\cap Q}$, but neither  $\{a\}\in\mathcal{P\cap Q}$ nor $\{b\}\in\mathcal{P\cap Q}.$
\end{example}

\begin{remark}
The family of sets formed by the intersection (union) of the elements of two primals need not be a primal on $X$ as shown by the following example.
\end{remark}

\begin{example}
Let $X=\{a,b\}$ with primals   $\mathcal{P} = \{\emptyset, \{a\}\}$ and $\mathcal{Q} = \{\emptyset, \{b\} \}$ on $X.$ 

$(a)$ The family $\mathcal{R}=\{P\cap Q|(P\in \mathcal{P})(Q\in\mathcal{Q})\}=\{\emptyset\}$ is not a primal on $X$ since $\{a\}\cap\{b\}=\emptyset\in \mathcal{R}$ but neither  $\{a\}\in\mathcal{R}$ nor $\{b\}\in\mathcal{R}.$

$(b)$ The family $\mathcal{S}=\{P\cup Q|(P\in \mathcal{P})(Q\in\mathcal{Q})\}=2^X$ is not a primal on $X$, since $X\in\mathcal{S}.$
\end{example}

\begin{definition}
A topological space $(X,\tau)$ with a primal $\mathcal{P}$ on $X$ is called a primal topological space and denoted by $(X,\tau,\mathcal{P}).$
\end{definition}

\begin{example}
Let $(X, \tau)$ be a topological space, where $X=\{a,b\}$, $\tau=\{\emptyset, \{a\}, X\}$ and    $\mathcal{P} = \{\emptyset, \{a\}\}$ primal on $X.$ Then, $(X, \tau, \mathcal{P})$ is a primal topological space. 
\end{example}

Now, we define a new kind of topological operator based on primal. The new structure is given below.


\begin{definition}
Let $(X,\tau,\mathcal{P})$ be a primal topological space. We consider a map $(\cdot)^{\diamond}: 2^X \rightarrow 2^X$ as $A^{\diamond}(X,\tau,\mathcal{P})=\{x\in X : (\forall U\in \tau (x))(A^c\cup U^c \in \mathcal{P})\}$ for any subset $A$ of $X.$ 
We can also write $A_{\mathcal{P}}^{\diamond}$ as $A^{\diamond}(X,\tau,\mathcal{P})$ to specify the primal as per our requirements.  
\end{definition}


\begin{remark}
Let $(X,\tau,\mathcal{P})$ be a primal topological space. For any subset $A$ of $X,$ $A^{\diamond}\subseteq A$ or $A\subseteq A^{\diamond}$ need not be true as shown by the following examples.
\end{remark}

\begin{example}
Let $X=\{a,b,c\}$ and $\tau=\{\emptyset,\{a\}, \{b\}, \{a,b\}, X\}$. We consider a primal  $\mathcal{P} = \{\emptyset, \{a\}, \{b\}, \{a,b\}\}$ on $X$. Now, if $A=\{a,b\},$ then  $ \emptyset=A^{\diamond} \subseteq A=\{a,b\}.$
\end{example}

\begin{example}
Let $X=\{a,b,c\}$ with the discrete topology. We consider a primal $\mathcal{P}=\{\emptyset,\{a\}, \{b\},\{a,b\}\}$ on $X.$ Now, if $A=\{c\},$ then $\{c\}=A\subseteq A^{\diamond}=X.$
\end{example}

\begin{theorem} \label{1}
Let $(X,\tau,\mathcal{P})$ be a primal topological space. Then, the following statements hold for any two subsets of $A$ and $B$ of $X.$\\

$(i)$ If $A^c\in \tau,$ then $ A^{\diamond}\subseteq A,$\\

$(ii)$ $\emptyset^{\diamond}= \emptyset$,\\

$(iii)$ $cl(A^{\diamond})=A^{\diamond},$ \\

$(iv)$ $(A^{\diamond})^{\diamond}\subseteq A^{\diamond},$ \\

$(v)$ If $A\subseteq B$, then $A^{\diamond}\subseteq B^{\diamond}$,\\

$(vi)$ $A^{\diamond}\cup B^{\diamond}= (A\cup B)^{\diamond},$\\


$(vii)$ $ (A\cap B)^{\diamond}\subseteq A^{\diamond}\cap B^{\diamond}.$\\
\end{theorem}

\begin{proof}

$(i)$ Let $A^c\in\tau$ and  $x\in A^{\diamond}.$ Suppose that  $x\notin A.$ Then, $A^c\in \tau(x).$ Since $x\in A^{\diamond},$ $A^c\cup U^c\in\mathcal{P}$ for all $U\in\tau(x).$ Therefore, $X=A\cup A^c=(A^c)^c\cup A^c\in \mathcal{P}.$ This contradicts with $X\notin \mathcal{P}.$ Hence,  $ A^{\diamond}\subseteq A.$\\

$(ii)$ It obvious from (i) since $\emptyset^c\in\tau.$ \\

$(iii)$ We have always $A^{\diamond}\subseteq cl(A^{\diamond}).$ Conversely, let $x\in cl(A^{\diamond})$ and $U\in\tau(x).$ 
Then, $U\cap A^{\diamond}\neq \emptyset.$ Therefore,  there exists $y\in X$ such that $y\in U$ and $y\in A^{\diamond}.$ Then, we have $V^c\cup A^c\in\mathcal{P}$ for all $V\in\tau(y).$ Thus, we get $U^c\cup A^c\in \mathcal{P}.$ This means that $x\in A^{\diamond}.$ Hence, $cl(A^{\diamond})\subseteq A^{\diamond}.$ Therefore, $A^{\diamond}$ is closed in $X.$
\\

$(iv)$ It is obvious from (i) and (iii).\\

$(v)$ Let $A\subseteq B$ and $x\in A^{\diamond}.$
Then, we have $A^c\cup U^c\in \mathcal{P}$ for all $U\in \tau(x).$ Thus, $B^c\cup U^c\in \mathcal{P}$ since $A\subseteq B.$ Hence, $x\in A^{\diamond}$. Hence, $A^{\diamond}\subseteq B^{\diamond}.$\\

$(vi)$ We can get from (v) that $A^{\diamond}\subseteq (A\cup B)^{\diamond}$ and $B^{\diamond}\subseteq (A\cup B)^{\diamond}$. Hence, $A^{\diamond}\cup B^{\diamond}\subseteq (A\cup B)^{\diamond}.$ Conversely, let $x\notin A^{\diamond}\cup B^{\diamond}.$ Then $x\notin A^{\diamond}$ and $x\notin B^{\diamond}.$ Then, there exist open sets $U$ and $V$ containing $x$ such that $A^c\cup U^c\notin \mathcal{P}$ and $B^c\cup V^c\notin \mathcal{P}.$ Put $W=U\cap V.$ Hence, $W$ is open containing $x$ such that $A^c\cup W^c\notin \mathcal{P}$ and $B^c\cup W^c\notin \mathcal{P}.$ Then, we have $(A\cup B)^c\cup W^c=(A^c\cap B^c)\cup W^c=(A^c\cup W^c)\cap (B^c\cup W^c)\notin \mathcal{P}$ since $\mathcal{P}$ is primal. This means that $x\notin (A\cup B)^{\diamond}.$ Thus, $(A\cup B)^{\diamond}\subseteq A^{\diamond}\cup B^{\diamond}.$
\\

$(vii)$ and $(viii)$ can be proven similarly.
\end{proof}

\begin{remark}
The inclusion given in (vii) of Theorem \ref{1} need not be reversible as shown by the following example.
\end{remark}


\begin{example}
Let $(X,\tau)$ be indiscrete topological space, where $X=\{a,b,c\}.$ We consider a primal  $\mathcal{P} = \{\emptyset, \{a\}, \{b\},\{c\},\{a,b\},\{a,c\}\}$ on $X$. Now, if $A=\{b\}$ and $B=\{c\},$ then we have $A^{\diamond}\cap B^{\diamond}=X\cap X=X \neq \emptyset =\emptyset^{\diamond}=(A\cap B)^{\diamond}.$
\end{example}

\begin{theorem} 
Let $(X,\tau,\mathcal{P})$ be a primal topological space and $A,B\subseteq X.$ If $A$ is open in $X,$ then $A\cap B^{\diamond}\subseteq (A\cap B)^{\diamond}.$
\end{theorem}

\begin{proof}
Let $A\in \tau$ and $x\in A\cap B^{\diamond}.$ Therefore, $x\in A$ and $x\in B^{\diamond}.$ Then, we have $B^c\cup U^c\in\mathcal{P}$ for all $U\in\tau(x).$ Since $A\in\tau,$ we get $(A\cap B)^c\cup U^c=B^c\cup (A\cap U)^c\in\mathcal{P}$ for all $U\in\tau(x).$ This means that $x\in (A\cap B)^{\diamond}.$ Thus, $A\cap B^{\diamond}\subseteq (A\cap B)^{\diamond}.$
\end{proof}

\begin{definition}
Let $(X,\tau,\mathcal{P})$ be a primal topological space. We consider a map $cl^{\diamond}: 2^X\rightarrow 2^X$ as $cl^{\diamond}(A)=A\cup A^{\diamond}$, where $A$ is any subset of $X$. 
\end{definition}

\begin{theorem} \label{2}
Let $(X,\tau,\mathcal{P})$ be a primal topological space and $A,B\subseteq X.$ Then, the following statements hold: \\

$(i)$ $cl^{\diamond}(\emptyset)= \emptyset$,\\

$(ii)$ $cl^{\diamond}(X)= X$,\\

$(iii)$ $A\subseteq cl^{\diamond}(A)$,\\

$(iv)$ If $A\subseteq B$, then $cl^{\diamond}(A)\subseteq cl^{\diamond}(B)$,\\

$(v)$ $cl^{\diamond}(A)\cup cl^{\diamond}(B)= cl^{\diamond}(A\cup B),$\\

$(vi)$ $cl^{\diamond}(cl^{\diamond}(A))=cl^{\diamond}(A).$
\\
\end{theorem}

\begin{proof}
Let $A,B\subseteq X.$\\

$(i)$ Since $cl^{\diamond}(\emptyset)=\emptyset,$ we have $cl^{\diamond}(\emptyset)=\emptyset\cup \emptyset^{\diamond}=\emptyset.$  
\\

$(ii)$ Since $X\cup X^{\diamond}=X,$ we have $cl^{\diamond}(X)=X.$\\

$(iii)$ Since $cl^{\diamond}(A)=A\cup A^{\diamond},$ we have $A\subseteq cl^{\diamond}(A).$\\

$(iv)$ Let $A\subseteq B.$ We get from (v) of Theorem \ref{1} that $A^{\diamond}\subseteq B^{\diamond}.$ Therefore, we have
$A\cup A^{\diamond}\subseteq B\cup B^{\diamond}$ which means that $cl^{\diamond}(A)\subseteq cl^{\diamond}(B).$
\\

$(v)$ It is obvious from the definition of the operator $cl^{\diamond}$ and (v) of  Theorem  \ref{1}.
\\

$(vi)$ It is obvious from (iii) that $cl^{\diamond}(A)\subseteq cl^{\diamond}(cl^{\diamond}(A)).$ On the other hand, since $A^{\diamond}$ is closed in $X,$ we have  $(A^{\diamond})^{\diamond}\subseteq A^{\diamond}.$ Therefore,

$$\begin{array}{rcl}
cl^{\diamond}(cl^{\diamond}(A)) & = & cl^{\diamond}(A)\cup \left(cl^{\diamond}(A)\right)^{\diamond} \\ & = & cl^{\diamond}(A)\cup (A\cup A^{\diamond})^{\diamond} \\ & = & cl^{\diamond}(A)\cup A^{\diamond}\cup (A^{\diamond})^{\diamond}
\\ & \subseteq  & cl^{\diamond}(A)\cup A^{\diamond}\cup A^{\diamond} \\ & = & cl^{\diamond}(A)
\end{array}$$
Thus, we have $cl^{\diamond}(cl^{\diamond}(A))=cl^{\diamond}(A).$
\end{proof}

\begin{corollary}
Let $(X,\tau,\mathcal{P})$ be a primal topological space. Then the function $cl^{\diamond}: 2^X\rightarrow 2^X$ defined by $cl^{\diamond}(A)=A\cup A^{\diamond}$, where $A$ is any subset of $X,$ is a Kuratowski's closure operator.
\end{corollary}

\begin{definition}
Let $(X,\tau,\mathcal{P})$ be a primal topological space. Then, the family $\tau^{\diamond}=\{A\subseteq X|cl^{\diamond}(A^c)=A^c\}$ is a topology on $X$ induced by topology $\tau$ and primal $\mathcal{P}.$  It is called primal topology on $X.$ We can also write $\tau_{\mathcal{P}}^{\diamond}$ instead of $\tau^{\diamond}$ to specify the primal as per our requirements.
\end{definition}

\begin{theorem} \label{tausubset}
Let $(X,\tau,\mathcal{P})$ be a primal topological space. Then the primal topology $\tau^{\diamond}$ is finer than $\tau.$
\end{theorem}

\begin{proof}
Let $A\in\tau.$ Then, $A^c$ is $\tau$-closed in $X.$ From (i) of Theorem \ref{1}, we get  $(A^c)^{\diamond}\subseteq A^c.$ Thus, $cl^{\diamond}(A^c)=A^c\cup (A^c)^{\diamond}\subseteq A^c.$ Since $A^c\subseteq cl^{\diamond}(A^c)$ is always true for any subset $A$ of $X,$ we obtain $cl^{\diamond}(A^c)=A^c.$ This means that $A\in\tau^{\diamond}.$  Thus, we have $\tau \subseteq \tau^{\diamond}.$
\end{proof}

\begin{theorem}\label{karsit}
Let $(X,\tau,\mathcal{P})$ be a primal topological space. Then, the following statements hold:\\

$(i)$ if $\mathcal{P}=\emptyset,$ then $\tau^{\diamond}=2^X,$\\

$(ii)$ if $\mathcal{P}=2^X\setminus\{X\},$ then $\tau=\tau^{\diamond}.$\\
\end{theorem}

\begin{proof}
$(i)$ We have always $\tau^{\diamond}\subseteq 2^X.$ Now, let $A\in 2^X.$ Since $\mathcal{P}=\emptyset,$ we have $A^{\diamond}=\emptyset$ for any subset $A$ of $X.$ Therefore, $cl^{\diamond}(A^c)=A^c.$ This means that $A\in\tau^{\diamond}.$ Hence, $2^X\subseteq\tau^{\diamond}.$ Thus, we have $\tau^{\diamond}=2^X.$ \\
$(ii)$ We have always $\tau\subseteq\tau^{\diamond}$ from Theorem \ref{tausubset}.  Now, we will prove that $\tau^{\diamond}\subseteq\tau.$ Let $A\in\tau^{\diamond}.$ Then  $A^c\cup (A^c)^{\diamond}= A^c$ which means that $(A^c)^{\diamond}\subseteq  A^c.$ Now, let $x\notin (A^c)^{\diamond}.$ Then, there exists $U\in\tau(x)$ such that $U^c\cup (A^c)^c=U^c\cup A\notin \mathcal{P}.$ Since $\mathcal{P}=2^X\setminus\{X\},$ we obtain $U^c\cup A=X$ and so $U\cap A^c=\emptyset.$ Therefore, $x\notin cl(A^c).$ Thus, we have $cl(A^c)\subseteq (A^c)^{\diamond}\subseteq  A^c.$ Hence, $cl(A^c)=A^c$ which means that $A^c$ is $\tau$-closed and so, $A\in\tau.$ Thus, $\tau^{\diamond}\subseteq \tau.$ Consequently, we have $\tau=\tau^{\diamond}.$
\end{proof}

\begin{remark}
The converse of Theorem \ref{karsit} need not be true as shown by the following example. 
\end{remark}

\begin{example}
Let $X=\{a,b,c\}$ with topology $\tau=\{\emptyset,X,\{a,b\}\}$ and $\mathcal{P}=2^X\setminus \{X,\{a,b\}\}.$ Simple calculations show that $\tau=\tau^{\diamond}$, but $\mathcal{P}$ is not equal to $2^X\setminus\{X\}.$
\end{example}

\begin{theorem}\label{3}
Let $(X,\tau,\mathcal{P})$ be a primal topological space and $A\subseteq X.$ Then the followings hold:\\

$(i)$ $A\in\tau^{\diamond}$ if and only if for all $x$ in $A,$ there exists an open set $U$ containing $x$ such that $U^c\cup A\notin\mathcal{P},$\\

$(ii)$ if $A\notin \mathcal{P},$ then $A\in\tau^{\diamond}.$
\end{theorem}

\begin{proof} $(i)$ Let $A\in\tau^{\diamond}.$
$$\begin{array}{rcl}A\in\tau^{\diamond}& \Leftrightarrow & cl^{\diamond}(A^c)=A^c \\
& \Leftrightarrow & A^c\cup (A^c)^{\diamond}=A^c
\\
& \Leftrightarrow & (A^c)^{\diamond}\subseteq A^c
\\
& \Leftrightarrow & A\subseteq ((A^c)^{\diamond})^c
\\
& \Leftrightarrow & (\forall x\in A)(x\notin (A^c)^{\diamond})
\\
& \Leftrightarrow & (\forall x\in  A)(\exists U\in\tau (x))(U^c\cup (A^c)^c=U^c\cup A\notin\mathcal{P}).
\end{array}$$

$(ii)$ Let $A\notin\mathcal{P}$ and $x\in A.$ Put $U=X.$ Then, $U$ is a $\tau$-open set containing $x.$ Since $A\notin \mathcal{P}$ and $U^c\cup A=A,$ we have $U^c\cup A\notin\mathcal{P}.$ From (i), we get $A\in\tau^{\diamond}.$
\end{proof}

\begin{theorem} 
Let $(X,\tau,\mathcal{P})$ be a primal topological space. Then, the family $\mathcal{B}_{\mathcal{P}}=\{T\cap P|T\in\tau$ and $ P\notin \mathcal{P}\}$ is a base for the primal topology $\tau^{\diamond}$ on $X.$
\end{theorem}

\begin{proof}
Let $B\in\mathcal{B}_{\mathcal{P}}.$ Then, there exist $T\in\tau$ and $P\notin\mathcal{P}$ such that $B=T\cap P.$ Since $\tau\subseteq\tau^{\diamond},$ we get $T\in\tau^{\diamond}.$ On the other hand, from Theorem \ref{1} (ii), we have $P\in\tau^{\diamond}.$ Therefore, $B\in\tau^{\diamond}.$ Consequently, $\mathcal{B}_{\mathcal{P}}\subseteq \tau^{\diamond}.$ Now, let $A\in\tau^{\diamond}$ and $x\in A.$ Then, from Theorem \ref{3} (i), there exists $U\in\tau(x)$ such that $U^c\cup A\notin\mathcal{P}.$ Now, let $B=U\cap (U^c\cup A).$ Hence, we have $B\in\mathcal{B}_{\mathcal{P}}$ such that $x\in B\subseteq A.$
\end{proof}


\begin{theorem}
Let $(X,\tau,\mathcal{P})$ and $(X,\tau,\mathcal{Q})$ be two primal topological spaces. If $\mathcal{P}\subseteq \mathcal{Q},$ then $\tau_{\mathcal{Q}}^{\diamond}\subseteq \tau_\mathcal{P}^{\diamond}.$
\end{theorem}

\begin{proof}
Let $A\in \tau_{\mathcal{Q}}^{\diamond}.$ Then  $A^c\cup(A^c)_{\mathcal{Q}}^{\diamond}= A^c$ which means that $(A^c)_{\mathcal{Q}}^{\diamond}\subseteq  A^c.$ Now, let $x\notin A^c.$ Then, we get $x\notin (A^c)_{\mathcal{Q}}^{\diamond}$ and so there exists $U\in\tau(x)$ such that $U^c\cup (A^c)^c=U^c\cup A\notin \mathcal{Q}.$ Since $\mathcal{P}\subseteq \mathcal{Q},$ we have $U^c\cup A\notin \mathcal{P}.$ Therefore, $x\notin (A^c)_{\mathcal{P}}^{\diamond}.$ Thus, $(A^c)_{\mathcal{P}}^{\diamond}\subseteq A^c$ and so  $cl^{\diamond}(A^c)=A^c\cup (A^c)_{\mathcal{P}}^{\diamond}=A^c.$ Hence, $A\in \tau_{\mathcal{P}}^{\diamond}.$  Consequently, we have $\tau_{\mathcal{Q}}^{\diamond}\subseteq\tau_{\mathcal{P}}^{\diamond}.$
\end{proof}

\begin{theorem} \label{teo} Let $f:X\to Y$ be a function and $\mathcal{P}\subseteq 2^X.$ If $\mathcal{P}$ is a primal on $X$ and $f$ is not surjective, then $\mathcal{Q} =\{f(P) | P\in \mathcal{P}\}$ is a primal on $Y$.
\end{theorem}

\begin{proof}
$(i)$ Suppose that $Y\in\mathcal{Q}.$ Then there exists $P\in\mathcal{P}$ such that $f(P)=Y.$ However, this contradicts the fact that $f$ is not surjective.
\\
\\

\color{black} $(ii)$ Let $A\in\mathcal{Q}$ and $B\subseteq A.$ Then, there exists $P_1\in\mathcal{P}$ such that $A=f(P_1).$ Now, set $P_2=f^{-1} (B)\cap P_1.$ It is obvious that $P_2\subseteq P_1.$ Since $\mathcal{P}$ is a primal on $X,$ $\mathcal{P}$ is downward closed and so we have $P_2\in\mathcal{P}.$ Also, $B=f(P_2).$ This means that $B\in\mathcal{Q}.$ 
\\

$(iii)$ Let $A\cap B\in\mathcal{Q}.$ Then, there exists $P_1,P_2\in\mathcal{P}$ such that $A=f(P_1)$ and $B=f(P_2).$ Since $\mathcal{P}$ is a primal on $X$ and $P_1\cap P_2\subseteq P_1,$ we have $P_1\cap P_2\in \mathcal{P}$ and so $P_1\in \mathcal{P}$ or $P_2\in \mathcal{P}.$ Therefore, $A=f(P_1)\in \mathcal{Q}$ or $B=f(P_2)\in \mathcal{Q}.$ 
\end{proof}


\begin{corollary} \label{C}
It is obvious from Theorem \ref{teo} that the property of being primal is not a topological property.
\end{corollary}

\begin{remark}\label{R}  For a function $f:X\to Y$ and a primal $\mathcal{Q}$ on $Y,$ the family $\mathcal{P} =\{f^{-1}(Q) | Q\in \mathcal{Q}\}$ need not be a primal on $X$ as shown by the following example.
\end{remark}

\begin{example}
Consider $X=\{a,b\},$ $Y=\{1,2\}$ and $\mathcal{Q}=\{\emptyset,\{1\}\}.$ Define the function $f:X\to Y$ by $f(x)=1.$ Then $\mathcal{Q}$ is a primal on $Y$ but $\mathcal{P}=\{f^{-1}(Q) | Q\in \mathcal{Q}\}=\{\emptyset,X\}$ is not a primal on $X.$ 
\end{example}

\section{Conclusion}

This paper introduced primal, which is a dual structure of grill. It is well known that ideal, dual structure of filter, is one of the highly useful notions in  topology \cite{Ja}, summability theory \cite{Bc}, real analysis \cite{Al}, etc. Thus, ideal inspired us to introduce primal. Here, we introduced two  new operators using primal. One  of these two operators satisfies Kuratowski's closure axioms. Moreover, we introduce a topology named primal topology ($\tau_\mathcal{P}$), which is finer than any topology $\tau$ of primal topological space $(X,\tau, \mathcal{P})$. Later, we provided structure of base for $\tau_\mathcal{P}$ and proved several fundamental results. It is important to note that primal structures and some related results showed quantum behaviours i.e. we could not determine some properties as universally true. For example,  one may refer examples 3.6 and 3.7. Moreover, Corollary \ref{C} and Remark \ref{R} showed some uncertain behaviours of this new structure. Thus, we hope that we should study this new notion more deeply in general topology and other areas. If possible, we are looking forward to connect this notion with some ideas of quantum world \cite{Is} from the perspective of general topology in future. \\

{\bf Conflict of interest:} The authors declare that there is no conflict of interest.

\end{document}